\documentclass[11pt]{article}
\usepackage[a4paper,margin=27mm]{geometry}
\usepackage[T1]{fontenc}
\usepackage{lmodern}
\usepackage{amsmath,amssymb,amsthm,microtype,fancyhdr}
\usepackage[colorlinks=true,linkcolor=blue,citecolor=blue,urlcolor=blue,
 pdftitle={A cyclic non-Ramsey heptagon},
 pdfauthor={D\"om\"ot\"or P\'alv\"olgyi}]{hyperref}
\newtheorem{theorem}{Theorem}
\newtheorem{lemma}[theorem]{Lemma}
\newcommand{\R}{\mathbb R}
\newcommand{\Q}{\mathbb Q}

\fancypagestyle{uncheckedappendix}{%
  \fancyhf{}
  \fancyhead[R]{\small Appendix by ChatGPT --- unchecked}
  \fancyfoot[C]{\thepage}
  
}
\begin{document}
\begin{center}
{\Large A cyclic non-Ramsey heptagon}\par
\medskip
D\"om\"ot\"or P\'alv\"olgyi\footnote{ELTE E\"otv\"os Lor\'and
University and Alfr\'ed R\'enyi Institute of Mathematics, Budapest,
Hungary.}
\end{center}
\vspace{3mm}

{\small\emph{AI disclosure.} This manuscript was written entirely by
	ChatGPT. I (D\"om\"ot\"or P\'alv\"olgyi) contributed nothing to the proofs or ideas, I only gave suggestions about the presentation.
	I also wrote some remarks about the proof, which can be found right before the Appendix, which contains further results by ChatGPT that I have not verified.\par}

\begin{abstract}
We present a construction with seven points on a circle that is not Ramsey, disproving a conjecture of Erd\H{o}s, Graham, Montgomery, Rothschild, Spencer and Straus.
\end{abstract}

A finite Euclidean set $P$ is \emph{Ramsey} if, for every positive integer
$k$, some $\R^n$ contains a monochromatic congruent copy of $P$ in every
$k$-colouring. The following disproves the conjecture of Erd\H{o}s,
Graham, Montgomery, Rothschild, Spencer and Straus \cite{EGMRSS}
that every finite spherical set is Ramsey.

\begin{theorem}\label{thm:main}
For any transcendental $r>2$, the seven-point set
\[
 P=\{(0,0)\}\;\cup\;
   \left\{\left(j,\pm\sqrt{j(2r-j)}\right):j\in\{1,3,4\}\right\}
\]
lies on the circle $(x-r)^2+y^2=r^2$ and is not Ramsey.
In particular, one may take $r=\pi$.
\end{theorem}

\begin{proof}
The circle equation is immediate, and $P$ cannot embed in $\R^1$.
We follow the strategy of EGMRSS \cite[proof of Theorem~13]{EGMRSS},
who excluded nonspherical sets using a fixed nonzero weighted sum of
squared norms. Here we seek fixed integer weights $\lambda_p$ with
$\sum_{p\in P}\lambda_p=0$ and, for each $n\ge2$, a function $F:\R^n\to\R$ such that,
for every congruent copy $P'=\{p':p\in P\}$ in $\R^n$,
\begin{equation}\label{eq:target}
                         \sum_{p\in P}\lambda_p F(p')=-24.
\end{equation}
The scalar equation $\sum_p\lambda_p z_p=-24$ has no constant solution.
By the inhomogeneous form of Rado's theorem
\cite[Satz XIII$'$, p.~470]{Rado1933}, with the
needed real-variable statement given in \cite[Lemma~15]{EGMRSS}, there
is a finite colouring $\chi$ of $\R$ with no monochromatic solution.
The colouring $\chi\circ F$ then avoids every congruent copy of $P$.
The same $\chi$ works in all dimensions. Thus it remains only to
construct these weights and the function $F$, which we now do.

A \emph{derivation} is an additive function $D:\R\to\R$ satisfying
$D(ab)=aD(b)+bD(a)$.
Since $D(1)=D(1\cdot1)=2D(1)$, we have $D(1)=0$, and hence $D(j)=0$
for every integer $j$.
To obtain $D(r)=1$, include $r$ in a transcendence basis of $\R/\Q$,
differentiate formally with respect to $r$ on the resulting rational
function field, and extend uniquely to its algebraic extension $\R$.
We apply $D$ entrywise to vectors and matrices.
For $p=(j,y)\in P\setminus\{(0,0)\}$, differentiating
$y^2=j(2r-j)$ gives
\[
 2yD(y)=D(y^2)=D\bigl(j(2r-j)\bigr)=2j,
 \qquad
 D(p)=\left(0,\frac{j}{y}\right).
\]

Assign weight $2$ to the origin and weights $-2,2,-1$ to each point
with first coordinate $1,3,4$, respectively. These choices come from
the multisets $\{0,3,3\}$ and $\{1,1,4\}$, which have equal cardinalities,
sums and sums of squares. Using the symmetry in $y$, we obtain
\begin{equation}\label{eq:moments}
\begin{gathered}
 \sum_p\lambda_p=0,\qquad
 \sum_p\lambda_p p=0,\qquad
 \sum_p\lambda_p pp^{\mathsf T}=0,\\
 \sum_p\lambda_p D(p)=0,\qquad
 \sum_p\lambda_p D(p)p^{\mathsf T}=0.
\end{gathered}
\end{equation}
All sums here and below are over $p\in P$.
On the other hand, setting $R=(2r-1)(2r-3)(2r-4)$ gives
\begin{equation}\label{eq:energy}
 \sum_p\lambda_p\|D(p)\|^2
   =-\frac4{2r-1}+\frac{12}{2r-3}-\frac8{2r-4}
   =-\frac{24}{R}.
\end{equation}

For $n\ge2$ and $q=(q_1,\ldots,q_n)\in\R^n$, define
\[
                         F(q)=R\sum_{i=1}^n D(q_i)^2
                               =R\|D(q)\|^2.
\]
Every congruent copy $P'$ of $P$ in $\R^n$ has the form
$p'=t+Ap$, where $A^{\mathsf T}A=I_2$. The product rule gives
\[
                         D(p')=D(t)+D(A)p+AD(p).
\]
Squaring the norm and summing with weights $\lambda_p$, grouped
according to the terms in this expression, gives
\[
\begin{aligned}
 \sum_p\lambda_p\|D(p')\|^2
 ={}& \|D(t)\|^2\sum_p\lambda_p\\
 &+2\left\langle D(t),
       D(A)\sum_p\lambda_p p+A\sum_p\lambda_pD(p)\right\rangle\\
 &+\sum_p\lambda_p\|D(A)p\|^2
   +2\sum_p\lambda_p\langle D(A)p,AD(p)\rangle\\
 &+\sum_p\lambda_p\|AD(p)\|^2.
\end{aligned}
\]
By \eqref{eq:moments}, all terms except the last vanish: the translation
terms use the zeroth and first moments, and the two terms involving
$D(A)p$ use the quadratic and mixed moments. Since $A$ is an isometric
embedding, \eqref{eq:energy} therefore yields
\[
 \sum_p\lambda_p F(p')
   =R\sum_p\lambda_p\|AD(p)\|^2
   =R\sum_p\lambda_p\|D(p)\|^2
   =-24,
\]
as required in \eqref{eq:target}.
\end{proof}

\clearpage

\paragraph{Author's note.}
I would like to make some closing remarks.
First, disproving the conjecture made by Erd\H{o}s, Graham, Montgomery, Rothschild, Spencer and
Straus raises the question of whether the ``rival'' conjecture, made by Leader, Russell, and Walters \cite{LeaderRussellWalters2012} might be true, according to which a set is Ramsey if and only if it is subtransitive, i.e., it is the subset of a finite set on which an isometry group acts in a transitive way.
I asked ChatGPT to prove this conjecture, but instead it showed (see Theorem~\ref{app:thm:dimbarrier} in the Appendix) that its method of finding what it calls fixed weighted derivation-energy identities cannot establish non-Ramseyness for spherical configurations having fewer than seven points, while we know that some cyclic quadrilaterals are not subtransitive \cite{LeaderRussellWalters2011}.
However, ChatGPT did find a proof that all transitive sets are Ramsey; in fact, it found that proof two weeks before this construction, when I asked it some follow-up questions regarding my latest paper.
I am still working on the exposition of that proof, which uses elementary group theory.

Finally, let me remark that the way the construction was found by ChatGPT was not surprising to me at all. In fact, I have been trying similar things myself (unsuccessfully, partly due to my limited computational capabilities, but also because I had never heard of derivations of fields) ever since I read a comment\footnote{ \url{https://mathoverflow.net/questions/300604/almost-monochromatic-point-sets\#comment748363_300604}, MathOverflow (20 May 2018).} by fedja\footnote{MathOverflow user fedja is semi-anonymous; see \url{https://meta.mathoverflow.net/questions/4351/how-to-cite-comment-by-unknown-user-disproving-erd\H os-conjecture?}} on MathOverflow, where he gave a field-automorphism based construction for a question about almost-monochromatic
configurations I asked.
As I later found out, Erd\H{o}s, Graham, Montgomery, Rothschild, Spencer and
Straus \cite{ErdosEtAl1975III} asked practically the same thing, so fedja's construction also disproved their Conjecture 4.

\paragraph{Acknowledgements.}
The author was supported by the NRDI EXCELLENCE--24 grant no.~151504,
Combinatorics and Geometry, and by the ERC Advanced Grant no.~101054936,
ERMiD.

\clearpage
\appendix
\pagestyle{uncheckedappendix}
\section*{Appendix}
\label{app:scope}
\setcounter{theorem}{0}
\renewcommand{\thetheorem}{A.\arabic{theorem}}
\renewcommand{\theHtheorem}{appendix.\arabic{theorem}}
\setcounter{equation}{0}
\renewcommand{\theequation}{A.\arabic{equation}}
\renewcommand{\theHequation}{appendix.\arabic{equation}}

\textbf{Further consequences and limitations of the derivation method}

\emph{The statements and proofs in this appendix were generated by
ChatGPT during research conversations with the author. They have not
been independently verified by the author and are included as unchecked
claims to support further investigation.}

We investigate both the scope and the limitations of the method in the
main proof. It applies to almost every seven-point subset of a circle
(Theorem~\ref{app:thm:generic}) and to additional families with algebraic
dependencies (Section~\ref{app:sec:seven}). There are also explicit examples
in every affine dimension at least two (Theorem~\ref{app:thm:sharp_dimension}),
and seven-point examples admitting an avoiding colouring with eight colours
(Theorem~\ref{app:thm:eight_colours}). Conversely, we obtain sharp point-count
restrictions for fixed weighted derivation-energy identities
(Theorems~\ref{app:thm:barrier} and~\ref{app:thm:dimbarrier}), an exact cubic
test for seven-point certificates (Theorem~\ref{app:thm:cubic_test}), and
obstructions that survive arbitrary quadratic finite iterates and many
nonlinear polynomial expressions in one derivation
(Theorems~\ref{app:thm:all-jets} and~\ref{app:thm:nonlinear-jet-barrier}).
A further obstruction applies to a broader class of colourings
(Theorem~\ref{app:thm:algebraic_translations}). These results do not produce
a non-Ramsey cyclic quadrilateral.

A derivation is an additive map $D:\R\to\R$ satisfying
$D(ab)=aD(b)+bD(a)$; it acts entrywise on vectors and matrices.
Throughout, $J(x,y)=(-y,x)$. The weights below may be arbitrary real
numbers: we never assume that $D$ annihilates them.

\section{The invariant and its linear algebra}

\begin{lemma}\label{app:lem:criterion}
Let $P=\{p_1,\ldots,p_m\}\subset\R^d$. Suppose real weights $\lambda_i$
and a derivation $D$ satisfy
\begin{equation}\label{app:eq:moments}
 \sum_i\lambda_i=0,\quad
 \sum_i\lambda_i p_i=0,\quad
 \sum_i\lambda_i p_i p_i^{\mathsf T}=0,\quad
 \sum_i\lambda_i D(p_i)=0,
\end{equation}
and the matrix $M=\sum_i\lambda_iD(p_i)p_i^{\mathsf T}$ is symmetric.
If $E=\sum_i\lambda_i\|D(p_i)\|^2\ne0$, then $P$ is not Ramsey.
\end{lemma}
\begin{proof}
For $n\ge d$, any congruent copy in $\R^n$ has the form
$p_i'=t+Ap_i$, where $A^{\mathsf T}A=I_d$. Expand $D(p_i')=D(t)+D(A)p_i+AD(p_i)$. Conditions
\eqref{app:eq:moments} cancel all terms except the original energy and
the mixed term
\[
 2\sum_i\lambda_i\langle D(A)p_i,AD(p_i)\rangle
       =2\operatorname{tr}\bigl(D(A)^{\mathsf T}AM\bigr)=0.
\]
The last equality holds because
$D(A)^{\mathsf T}A+A^{\mathsf T}D(A)=0$, so the first factor is
skew-symmetric and $M$ is symmetric. Thus
$\sum_i\lambda_i\|D(p_i')\|^2=E$. The same identity holds in
lower dimensions by adding zero coordinates.

The equation $\sum_i\lambda_i z_i=E$ has no constant solution.
The real-coefficient inhomogeneous colouring theorem
\cite[Lemma~15]{EGMRSS} gives a finite colouring $\chi$ of $\R$
avoiding monochromatic solutions. Pulling it back through
$q\mapsto\|D(q)\|^2$ proves the claim.
\end{proof}

For at least three distinct points $p_i=(x_i,y_i)$ on the unit circle,
define a symmetric
bilinear form on $\R^m$ and a three-dimensional subspace by
\[
 B(u,v)=\sum_i\lambda_i u_i v_i,\qquad
 W=\operatorname{span}\{\mathbf1,x,y\}.
\]
The first three conditions in \eqref{app:eq:moments} say exactly that
$B$ vanishes on $W\times W$, or $W\subseteq W^\perp$.
Differentiating $\|p_i\|^2=1$ gives
\[
 D(p_i)=b_iJp_i
\]
for some $b_i\in\R$. The remaining conditions of Lemma~\ref{app:lem:criterion}
are exactly
\begin{equation}\label{app:eq:angular}
 B(b,\mathbf1)=B(b,x)=B(b,y)=0,\qquad E=B(b,b)\ne0.
\end{equation}
Indeed, the derivative first moment gives the last two orthogonality
conditions, and $M_{21}-M_{12}=\sum_i\lambda_i b_i$ gives the first.
Thus the construction seeks a vector in $W^\perp$ with nonzero
squared length for $B$.

\section{Why six points cannot work}

\begin{theorem}\label{app:thm:barrier}
Let $P$ consist of at most six distinct points on a circle.
For any derivation $D$ and any fixed real weights $\lambda_p$, if
\[
                  \sum_{p\in P}\lambda_p\|D(p')\|^2=C
\]
holds for every congruent copy $P'$ in $\R^3$, then $C=0$.
\end{theorem}
\begin{proof}
The case $D=0$ is immediate, so assume $D\ne0$.
We first show that the ordinary moment conditions are necessary,
rather than merely sufficient. Choose $u$ with
$D(u)\ne0$, and put
\[
 c=\frac{1-u^2}{1+u^2},\quad s=\frac{2u}{1+u^2},\quad
 h=D(c)^2+D(s)^2=\frac{4D(u)^2}{(1+u^2)^2}>0.
\]
The two embeddings $(x,y)\mapsto(x,y,0)$ and
$(x,y)\mapsto(cx,y,sx)$ are isometric, and
\[
 \|D(cx,y,sx)\|^2=\|D(x,y)\|^2+h x^2,
\]
since $c^2+s^2=1$ and $cD(c)+sD(s)=0$.
The assumed identity therefore forces $\sum_p\lambda_p x_p^2=0$.
Apply the same comparison to every translate of $P$:
$\sum_p\lambda_p(x_p+a)^2=0$ for every $a\in\R$.
Repeating in the $y$ and $(x+y)/\sqrt2$ directions gives
\begin{equation}\label{app:eq:necessary}
 \sum_p\lambda_p=0,\qquad
 \sum_p\lambda_p p=0,\qquad
 \sum_p\lambda_p pp^{\mathsf T}=0.
\end{equation}

At any specified point among at most five distinct circle points,
a quadratic polynomial can be nonzero while vanishing at all the
others: take the product of two line equations covering the other
points and avoiding the specified point. Thus \eqref{app:eq:necessary}
forces every weight to vanish for at most five points.
For six points, we may therefore assume that all six weights are nonzero.

Translation invariance, together with $\sum_p\lambda_p=0$, gives
$\sum_p\lambda_p D(p)=0$. Planar rotation invariance also gives
\begin{equation}\label{app:eq:rotation}
                       \sum_p\lambda_p\langle Jp,D(p)\rangle=0.
\end{equation}
To see this, take $A=\left(\begin{smallmatrix}c&-s\\s&c\end{smallmatrix}\right)$.
Then $D(A)=A\omega J$, where $\omega=2D(u)/(1+u^2)\ne0$.
The difference between the energies of $AP$ and $P$ is
$2\omega\sum_p\lambda_p\langle Jp,D(p)\rangle
+\omega^2\sum_p\lambda_p\|p\|^2$, whose last term vanishes by
\eqref{app:eq:necessary}.

Translate the circle's centre to the origin and write $p_i=\rho z_i$,
where $z_i=(x_i,y_i)$ is on the unit circle.
Write $D(z_i)=b_iJz_i$.
The zero first moments and \eqref{app:eq:rotation} imply
$\sum_i\lambda_i b_i=\sum_i\lambda_i b_i x_i
=\sum_i\lambda_i b_i y_i=0$.
For $B$ and $W$ defined using these unit-circle coordinates, we have
$W\subseteq W^\perp$. All weights are nonzero, so $B$ is nondegenerate;
both spaces have dimension three, giving $W=W^\perp$.
Hence $b\in W$ and $B(b,b)=0$. Finally,
\[
 C=\sum_i\lambda_i\|(D\rho)z_i+\rho b_iJz_i\|^2
   =(D\rho)^2\sum_i\lambda_i+\rho^2 B(b,b)=0.
\]
\end{proof}

\section{Almost every seven-point circular set}

\begin{theorem}\label{app:thm:generic}
Put
\[
 p(t)=\left(\frac{1-t^2}{1+t^2},\frac{2t}{1+t^2}\right).
\]
Fix three distinct real numbers $t_1,t_2,t_3$.
If $t_4,t_5,t_6,t_7$ are algebraically independent over
$K=\Q(t_1,t_2,t_3)$, then $\{p(t_1),\ldots,p(t_7)\}$ is not Ramsey.
Consequently almost every seven-point subset of a fixed circle
is not Ramsey.
\end{theorem}
\begin{proof}
Put
\[
 H(T)=\prod_{i=1}^7(T-t_i),\qquad
 g(T)=(T-t_1)(T-t_2)(T-t_3).
\]
Polynomial interpolation gives
\begin{equation}\label{app:circle:interp}
 \sum_i\frac{t_i^k}{H'(t_i)}=0\quad(0\le k\le5),\qquad
 \sum_i\frac{t_i^6}{H'(t_i)}=1.
\end{equation}
Choose a derivation vanishing on $K$ with $D(t_i)=g(t_i)$ for
$4\le i\le7$. Algebraic independence allows these prescriptions;
extend to a transcendence basis and then algebraically to $\R$.
Since $g(t_i)=0$ for $i\le3$, the formula holds for all seven indices.
Writing $p(t_i)=(x_i,y_i)$, differentiation gives
\[
 D(p(t_i))=b_iJp(t_i),\qquad
 b_i=\frac{2g(t_i)}{1+t_i^2}.
\]
Take $\lambda_i=(1+t_i^2)^2/H'(t_i)$.
The ordinary zeroth, first and quadratic moments vanish by
\eqref{app:circle:interp}: after multiplication by $\lambda_i$,
each coordinate polynomial has the form $h(t_i)/H'(t_i)$ with
$\deg h\le4$. The remaining linear conditions are
\[
 \sum_i\lambda_i b_i(1,x_i,y_i)
   =2\sum_i\frac{g(t_i)}{H'(t_i)}(1+t_i^2,1-t_i^2,2t_i)=0,
\]
since each numerator has degree at most five. Finally, $g$ is a
monic cubic, so
\[
 E=\sum_i\lambda_i b_i^2
      =4\sum_i\frac{g(t_i)^2}{H'(t_i)}=4.
\]
Thus the circle criterion \eqref{app:eq:angular} and
Lemma~\ref{app:lem:criterion} apply.

Algebraic independence of all seven parameters over $\Q$ holds
outside a countable union of polynomial zero sets, a set of Lebesgue
measure zero. Stereographic parametrization transfers this conclusion
to arc-length measure on the circle; similarities give the statement
for any fixed circle.
\end{proof}

\newpage
\section{Several derivations give no stronger certificates}

The following reduction holds for arbitrary finite Euclidean sets,
not just for points on a circle.

\begin{theorem}\label{app:thm:single_reduction}
Let $P\subset\R^d$ be finite, let $D_1,\ldots,D_k$ be
derivations, and let $C=(c_{ab})$ be a real symmetric matrix. Put
\[
 F(q)=\sum_{a,b=1}^k c_{ab}
                   \langle D_a(q),D_b(q)\rangle.
\]
Suppose fixed real weights $\lambda_p$ and a nonzero constant $K$
satisfy $\sum_{p\in P}\lambda_pF(p')=K$ for every congruent
copy $P'$ in $\R^{d+1}$. Then there is a single derivation $D$
and a nonzero constant $E$ such that
\[
             \sum_{p\in P}\lambda_p\|D(p')\|^2=E
\]
for every congruent copy in every Euclidean dimension.
\end{theorem}

\begin{proof}
Replace the derivations by a basis of their real linear span and then
diagonalize the coefficient matrix, discarding its zero directions.
We may assume that $D_1,\ldots,D_k$ are linearly independent over
$\R$ and that $C$ is invertible. The vectors
$v(u)=(D_1u,\ldots,D_ku)$ span $\R^k$. Moreover,
$Q(u)=v(u)^{\mathsf T}Cv(u)$ is nonzero for some $u$: otherwise
polarizing $Q(u+w)=0$ would show that $C$ vanishes on their span.

For such a $u$, put $c=(1-u^2)/(1+u^2)$ and $s=2u/(1+u^2)$.
Tilting one coordinate $x$ into a new coordinate by
$(x,0)\mapsto(cx,sx)$ changes $F$ by exactly
\[
                         \frac{4Q(u)}{(1+u^2)^2}x^2.
\]
Comparison with the original copy, followed by translations and the
same comparison in the coordinate and pairwise diagonal directions,
therefore gives
\begin{equation}\label{app:eq:reduction_moments}
 \sum_p\lambda_p=0,\qquad \sum_p\lambda_p p=0,
 \qquad \sum_p\lambda_p pp^{\mathsf T}=0.
\end{equation}

Set $\mu_b=\sum_p\lambda_pD_b(p)$. Translation invariance and \eqref{app:eq:reduction_moments}
give, for every translation vector $t$,
\[
               \sum_{a,b}c_{ab}\langle D_a(t),\mu_b\rangle=0.
\]
Taking $t=u e_j$ and varying $u$, the vectors $v(u)$ span $\R^k$;
the invertibility of $C$ therefore gives $\mu_b=0$ for every $b$.

Next rotate any coordinate two-plane by the matrix with entries $c,s$
above. If $J$ denotes its infinitesimal rotation, then
$D_a(A)=A\omega_a J$, where
$\omega_a=2D_a(u)/(1+u^2)$. After \eqref{app:eq:reduction_moments} cancels the quadratic terms,
rotation invariance gives
\[
 \sum_{a,b}c_{ab}\omega_a
                     \sum_p\lambda_p\langle Jp,D_b(p)\rangle=0.
\]
As $u$ varies, the vectors $\omega(u)$ span $\R^k$, so each inner sum is zero.
Consequently, for every $b$, the matrix
$\sum_p\lambda_pD_b(p)p^{\mathsf T}$ is symmetric.

These are exactly the moment conditions that make
$\sum_p\lambda_p\|D_b(p')\|^2$ invariant under isometric embeddings:
the expansion in Lemma~\ref{app:lem:criterion} works identically in
$\R^d$. They also hold for every real linear combination of
the $D_b$. Finally, the symmetric matrix
\[
        H_{ab}=\sum_p\lambda_p\langle D_a(p),D_b(p)\rangle
\]
is nonzero, since $\sum_{a,b}c_{ab}H_{ab}=K\ne0$.
Choose $z$ with $z^{\mathsf T}Hz\ne0$ and set
$D=\sum_a z_aD_a$. Its invariant energy is
$E=z^{\mathsf T}Hz\ne0$, as required.
\end{proof}

In particular, the six-point obstruction on a circle remains unchanged
for every quadratic combination of finitely many derivations.

\section{Further seven-point constructions}\label{app:sec:seven}

\subsection*{Stereographic coordinates and M\"obius transformations}

Write
\[
 p(t)=\left(\frac{1-t^2}{1+t^2},\frac{2t}{1+t^2}\right).
\]
For distinct finite real parameters $t_1,\ldots,t_m$, put
$v_i=D(t_i)$ and $w_i=\lambda_i/(1+t_i^2)^2$.
The conditions of Lemma~\ref{app:lem:criterion} on the circle are
equivalent to
\begin{equation}\label{app:circle:moments}
 \sum_iw_it_i^k=0\quad(0\le k\le4),\qquad
 \sum_iw_iv_it_i^k=0\quad(0\le k\le2),\qquad
 \sum_iw_iv_i^2\ne0.
\end{equation}
Indeed, $D(p(t_i))=2v_iJp(t_i)/(1+t_i^2)$; the coordinates
$1,x,y$ span the quadratic polynomials divided by $1+t^2$, and their
products span the polynomials of degree at most four divided by
$(1+t^2)^2$. The corresponding energy is
$4\sum_iw_iv_i^2$.

\begin{lemma}[M\"obius invariance]\label{app:circle:mobius}
Existence of a certificate \eqref{app:circle:moments} is invariant under real
M\"obius transformations of the circle. The same derivation may be used.
\end{lemma}
\begin{proof}
Let $s_i=h(t_i)$, where
\[
 h(t)=\frac{at+b}{ct+d},\qquad \Delta=ad-bc\ne0,
\]
and first suppose all parameters before and after the transformation
are finite. The product rule gives
\[
 D(s_i)=\frac{\Delta v_i+g(t_i)}{(ct_i+d)^2}
\]
for a polynomial $g$ of degree at most two. This remains true when
$D$ does not annihilate the coefficients of $h$.
Choose
\[
 \widetilde w_i=\frac{w_i(ct_i+d)^4}{\Delta^2}.
\]
The first conditions of \eqref{app:circle:moments} for $s_i$ follow from those
for $t_i$, because $(at+b)^k(ct+d)^{4-k}$ has degree at most four.
For the mixed conditions, the corresponding expression is a multiple of
\[
 \sum_iw_i\bigl(\Delta v_i+g(t_i)\bigr)
              (at_i+b)^k(ct_i+d)^{2-k}\qquad(0\le k\le2),
\]
which vanishes by the first two groups of conditions. Finally,
\[
 \sum_i\widetilde w_iD(s_i)^2
 =\frac1{\Delta^2}\sum_iw_i\bigl(\Delta v_i+g(t_i)\bigr)^2
 =\sum_iw_iv_i^2.
\]
The cross term and the square of $g$ vanish by the same conditions.
Apply the argument also to $h^{-1}$ for the converse. Points at infinity
can be avoided by auxiliary circle rotations with rational matrix
entries. Such rotations preserve the original circle certificate
directly, so this reduction does not require the result being proved.
\end{proof}

The lemma concerns the existence of these certificates; it does not
assert that the Ramsey property itself is M\"obius invariant.

\subsection*{Allowing one algebraic relation}

The next result permits a relation involving all seven parameters.

\begin{theorem}\label{app:circle:raw}
If every six of seven distinct real numbers $t_1,\ldots,t_7$ are
algebraically independent over $\Q$, then $\{p(t_1),\ldots,p(t_7)\}$
is not Ramsey.
\end{theorem}
\begin{proof}
Put $H(T)=\prod_{i=1}^7(T-t_i)$. By the interpolation identities
\eqref{app:circle:interp}, any realizable velocity vector of the form
$v_i=t_i^3+q(t_i)$, where $\deg q\le2$, gives
\eqref{app:circle:moments} with $w_i=1/H'(t_i)$: the last sum is $1$.
If all seven parameters are algebraically independent, we may simply
prescribe $D(t_i)=t_i^3$ and extend the derivation to $\R$.

Otherwise the transcendence degree is six. Choose a nonzero polynomial
$f\in\Q[T_1,\ldots,T_7]$ of smallest total degree with
$f(t_1,\ldots,t_7)=0$, and put
$g_i=(\partial f/\partial T_i)(t_1,\ldots,t_7)$.
The polynomial $f$ involves each variable, because the other six
parameters are independent. Thus every partial derivative is a nonzero
polynomial of smaller total degree, and minimality gives $g_i\ne0$.
Any six parameters form a transcendence basis of their field, and the
seventh is separably algebraic over the corresponding rational function
field. Consequently the realizable vectors
$(D(t_1),\ldots,D(t_7))$ are precisely
\begin{equation}\label{app:circle:tangent}
                          \sum_i g_i v_i=0.
\end{equation}
For completeness, prescribe the six values on any transcendence basis
obtained by omitting one $t_i$, and extend to the remaining algebraic
parameter. Differentiating $f=0$ yields its unique value, because
$g_i\ne0$. Extend further to $\R$ by adjoining a transcendence basis
and then taking the unique algebraic extension of the derivation.

If there is a polynomial $q$ of degree at most two with
$\sum_i g_iq(t_i)\ne0$, a suitable vector
$v_i=t_i^3+\alpha q(t_i)$ satisfies \eqref{app:circle:tangent}, and we are done.
We may therefore suppose that
$\sum_i g_it_i^k=0$ for $0\le k\le2$.
If also $\sum_i g_it_i^3=0$, again take $v_i=t_i^3$.
Otherwise the polynomial
\[
 N(z)=\sum_i g_i\prod_{j\ne i}(t_j-z)
\]
has degree exactly three: its leading coefficients in degrees six,
five and four vanish by the three moment identities, while its cubic
coefficient is nonzero. Choose a real root $z$ of $N$.
For each $i$,
\[
 N(t_i)=g_i\prod_{j\ne i}(t_j-t_i)\ne0,
\]
so $z$ differs from all seven parameters. Hence
\[
 v_i=\frac1{t_i-z},\qquad
 w_i=\frac{t_i-z}{H'(t_i)}
\]
satisfy \eqref{app:circle:tangent} and the first two groups of
\eqref{app:circle:moments}, using \eqref{app:circle:interp}. Partial fractions give
\[
 \sum_iw_iv_i^2
 =\sum_i\frac1{(t_i-z)H'(t_i)}=-\frac1{H(z)}\ne0.
\]
The corresponding derivation and weights therefore yield a certificate.
\end{proof}

For $s\ge3$ distinct ordered points of the real projective line, their
\emph{projective moduli} are the $s-3$ values obtained after sending the
first three points to $\infty,0,1$ by a M\"obius transformation.
Their field's transcendence degree over $\Q$ is independent of the
ordering, because the coordinate changes are rational over $\Q$.
Circle points are identified with the projective line by stereographic
projection. In particular, saying that six circle points have three
algebraically independent projective moduli is unambiguous.

\begin{theorem}\label{app:circle:intrinsic}
Let $P$ be seven distinct points on a circle. If every six-point subset
of $P$ has three algebraically independent projective moduli over
$\Q$, then $P$ is not Ramsey.
\end{theorem}
\begin{proof}
Normalize the circle to the unit circle and choose a stereographic
chart avoiding the seven points. Let $K$ be the finitely generated
field of their seven parameters. Apply a M\"obius transformation whose
three parameters are algebraically independent over $K$; they may be
chosen real, with no pole at the seven points.

For any chosen six points, their transformed coordinates have
transcendence degree six over $\Q$. Here is an explicit way to see the
parameter count. The six coordinates are rationally equivalent to
their three projective moduli and the three coordinates of the first
three points. The moduli are unchanged, and the first three transformed
coordinates are algebraically independent over $K$, because sending
three fixed distinct points to three prescribed distinct images
uniquely determines a M\"obius transformation, rationally in those
images. Thus these six quantities have transcendence degree $3+3=6$.
Every six transformed parameters are therefore algebraically
independent. Apply Theorem~\ref{app:circle:raw} and transfer its certificate back
by Lemma~\ref{app:circle:mobius}.
\end{proof}

Unlike a condition asking for four independent moduli of the full
seven-point set, Theorem~\ref{app:circle:intrinsic} allows the full moduli to have
transcendence degree three. A single relation involving all seven
marked points is permitted, provided no relation survives on a
six-point subset.

\subsection*{An exact cubic test for seven points}

The preceding argument has an exact algebraic formulation. By the
M\"obius invariance just proved, normalize three stereographic parameters
to distinct algebraic numbers $a_1,a_2,a_3$, with all seven parameters
finite. Write the remaining parameters as $s_1,\ldots,s_4$, and put
$A(T)=\prod_{j=1}^3(T-a_j)$. Define the real vector space
\[
 V=\left\{v\in\R^4:
       \sum_{i=1}^4 \frac{\partial f}{\partial X_i}(s)v_i=0
       \text{ whenever }f\in\Q[X_1,\ldots,X_4],\ f(s)=0\right\}.
\]
Equivalently, $V$ consists of all possible vectors
$(D(s_1),\ldots,D(s_4))$. Its dimension is the transcendence degree
$r$ of the projective moduli. The evaluated gradients of generators of
the ideal of algebraic relations span $V^\perp$, and $V$ is their common
kernel.

\begin{theorem}\label{app:thm:cubic_test}
For each vector $g$ in a basis of $V^\perp$, form the binary cubic
\[
 C_g(Z,W)=\sum_{i=1}^4g_i A(s_i)
                           \prod_{\substack{1\le j\le4\\j\ne i}}
                                      (Z-s_jW).
\]
The seven-point set admits a nonzero fixed weighted derivation-energy
certificate if and only if these cubics have a common real projective
zero different from the seven marked parameters.
If $r=4$, the condition is automatic. If $r<4$, there are at most $r$
candidate projective zeros; in particular, $r=3$ requires testing the
roots of a single explicit cubic. The same criterion applies to finite
quadratic combinations of derivations.
\end{theorem}

\begin{proof}
For seven distinct finite parameters $t_i$, let
$H(T)=\prod_i(T-t_i)$. The five moment conditions give
$w_i=(at_i+b)/H'(t_i)$. Every weight must be nonzero, by the six-point
barrier. The mixed conditions then say
\[
             D(t_i)=\frac{q(t_i)}{at_i+b},\qquad \deg q\le3.
\]
Modulo quadratic-polynomial velocities, a nonzero-energy solution is
therefore either a nonzero multiple of $1/(t_i-z)$, with $z$ not a
marked parameter, or a cubic-polynomial velocity, corresponding to
$z=\infty$. Indeed, for finite $z$, writing
$D(t_i)=q_2(t_i)+\kappa/(t_i-z)$ gives
$\sum_iw_iD(t_i)^2=-a\kappa^2/H(z)$; for constant $at+b$, the
energy is a nonzero scalar times the square of the cubic coefficient.

Since $D(a_j)=0$, subtract from $1/(T-z)$ its quadratic interpolant
at the three $a_j$. The resulting velocity at $s_i$ is
\[
                        \frac{A(s_i)}{A(z)(s_i-z)}.
\]
Its projective class is represented by
\[
 u_i(Z,W)=A(s_i)\prod_{j\ne i}(Z-s_jW),
\]
and $u_i(1,0)=A(s_i)$ is the cubic-polynomial direction. Thus an
allowable velocity exists precisely when $u(Z,W)\in V$, equivalently
when every $C_g(Z,W)$ vanishes, at a pole other than the seven marked
parameters. Conversely each such vector in $V$ is realized by a
derivation: define it first on $\Q(s_1,\ldots,s_4)$ using the
differentiated relations, extend a transcendence basis to one for
$\R$, and extend through the algebraic extension in characteristic
zero. This proves both implications.

The four polynomials $u_i$ form a basis of binary cubics, since
evaluation at $(s_i,1)$ isolates the $i$th one. Their projective image
is consequently a rational normal cubic. Any $r+1\le4$ distinct
points on this curve are linearly independent, so the $r$-dimensional
space $V$ contains at most $r$ candidate directions. The final claim
follows from Theorem~\ref{app:thm:single_reduction}.
\end{proof}

The common-zero condition can equivalently be checked by taking the
homogeneous greatest common divisor of the cubics and inspecting its
real roots. This characterizes certificates, not Ramsey sets.

\subsection*{Why three independent moduli alone do not suffice}

The preceding test also shows that transcendence degree three is the
generic threshold for this certificate method.  Indeed, when \(r=3\),
\(\mathbb P(V)\) is a real projective plane and its equation restricts
to a real binary cubic on the rational normal cubic, so it has a real
projective zero.  Generically this zero is not marked and gives a
certificate.  For \(r=2\), by contrast, a generic projective line in
\(\mathbb P^3\) misses the cubic.  These are statements about the
certificate method, not about the Ramsey property.

More precisely, for fixed marked parameters the exceptional planes can
be described completely. Writing $L_m$ for a linear form vanishing at a
marked parameter $m$, their binary cubics are exactly $L_mQ$ with $Q$ a
definite real quadratic, together with the finitely many products
$L_{m_1}L_{m_2}L_{m_3}$ of marked factors, with repetition allowed.
Hence all exceptional planes lie in the union of the seven incidence
planes $C_g(m)=0$, while their complement contains a Zariski-open dense
set. Indeed, every real cubic has a real projective zero; it is
exceptional precisely when all its real zeros are marked.

There are nevertheless seven-point configurations with three
independent projective moduli but no certificate of the above kind.
Normalize the three algebraic parameters to
\[
                         -2,\quad 0,\quad 2,
\]
and choose \(s_1,s_2,s_4\) algebraically independent over \(\Q\) and
sufficiently close to \(-3,-1,3\), respectively.  Put
\[
                         s_3=-4-s_1-2s_2.
\]
Then \(V^\perp\) is spanned by \(g=(1,2,1,0)\).  With
\(A(T)=(T+2)T(T-2)\), the cubic from
Theorem~\ref{app:thm:cubic_test}, at the base point
\((s_1,s_2,s_3,s_4)=(-3,-1,1,3)\), is
\[
                         C_g(z,1)=-12(z-3)(z^2+1).
\]
For example, the coefficients of \((z-3)(z^2+1)\) in the Lagrange
basis \(A(s_i)\prod_{j\ne i}(z-s_j)\) are
\(-1/12,-1/6,-1/12,0\).

Since \(g_4=0\), throughout this family the cubic has the factor
\(z-s_4\).  Its remaining real quadratic factor has negative
discriminant near the base point.  Thus its only real projective zero
is the forbidden marked pole \(s_4\), and there is no nonzero
certificate, even from a finite quadratic combination of derivations.

Exactly one six-point subset is nongeneric.  Deleting \(s_4\) leaves
the displayed relation, whereas deleting any of \(s_1,s_2,s_3\)
leaves three independent moduli.  For the three anchor deletions, the
base cubic has nonzero values \(300,36,60\) at \(-2,0,2\), respectively,
so the corresponding forgetful maps have full differential rank
nearby.  This example does not settle the Ramsey status of the
configuration.

\subsection*{One-parameter families and a limitation}

\begin{theorem}[One-parameter families]\label{app:thm:one_parameter}
Let $r$ be transcendental, and let $c_1,\ldots,c_7$ be distinct real
algebraic numbers with $r+c_i>0$. Then
\[
 \left\{p\bigl(\sqrt{r+c_i}\bigr):1\le i\le7\right\}
\]
is not Ramsey.
\end{theorem}
\begin{proof}
Put $t_i=\sqrt{r+c_i}$ and $H(T)=\prod_i(T-t_i)$.
A derivation with $D(r)=1$ gives $D(t_i)=1/(2t_i)$.
Take $w_i=t_i/H'(t_i)$ in \eqref{app:circle:moments}.
The first two groups of conditions follow from \eqref{app:circle:interp}, and
\[
 \sum_iw_iD(t_i)^2
 =\frac14\sum_i\frac1{t_iH'(t_i)}
 =-\frac1{4H(0)}\ne0.
\]
\end{proof}

\begin{theorem}[At most three nonalgebraic points]\label{app:circle:sparse}
Suppose a M\"obius transformation sends all but at most three points
of a finite subset $P$ of the unit circle to points with algebraic
coordinates. Then $P$ admits no certificate of the circle derivation
criterion.
\end{theorem}
\begin{proof}
By Lemma~\ref{app:circle:mobius}, it is enough to consider the transformed set.
For any derivation $D$, write $D(p_i)=b_iJp_i$.
Derivations annihilate algebraic numbers, so $b_i=0$ at every point
with algebraic coordinates. The linear conditions in the circle
criterion give
\[
                    \sum_i\lambda_i b_i(1,x_i,y_i)=0.
\]
At most three summands are nonzero, and the corresponding vectors
$(1,x_i,y_i)$ are linearly independent: three distinct circle points
are not collinear. Thus $\lambda_i b_i=0$ for each $i$, and the energy
$\sum_i\lambda_i b_i^2$ is zero.
\end{proof}

The same limitation applies to finite quadratic combinations of
derivations, by Theorem~\ref{app:thm:single_reduction}. This is a limitation of the certificate method and does
not determine whether the configuration is Ramsey.

\section{Higher-dimensional spheres}

\subsection*{A dimension-dependent limitation}

The moment conditions of Lemma~\ref{app:lem:criterion} are also
necessary for a fixed weighted identity in $\R^{d+1}$, provided $D\ne0$.
Indeed, the tilt used above can be applied to any coordinate direction
and any translate of $P$, and to the directions $(e_j+e_k)/\sqrt2$.
It gives the zeroth, first and quadratic moment conditions.
Translation by an integer multiple of $u e_j$, where $D(u)\ne0$,
then gives $\sum_i\lambda_iD(p_i)=0$.
Finally, rotation in any coordinate two-plane through the angle with
cosine $(1-u^2)/(1+u^2)$ and sine $2u/(1+u^2)$ shows that
$\sum_i\lambda_iD(p_i)p_i^{\mathsf T}$ is symmetric: its skew part
is the only surviving term in the difference of the energies.

\begin{theorem}\label{app:thm:dimbarrier}
Suppose $P=\{p_1,\ldots,p_m\}$ is spherical and affinely spans
$\R^d$, and all weights $\lambda_i$ are nonzero.
If $m\le 2d+2$ and
\[
 \sum_i\lambda_i\|D(p_i')\|^2=C
\]
for every congruent copy of $P$ in $\R^{d+1}$, then $C=0$.
Consequently no spherical set of at most six points admits a nonzero
fixed weighted identity of this form, in any dimension.
\end{theorem}
\begin{proof}
Assume $D\ne0$, and put
$B(a,b)=\sum_i\lambda_i a_i b_i$ and
$W=\operatorname{span}\{\mathbf1,p^{(1)},\ldots,p^{(d)}\}$, where
$p^{(j)}$ is the vector of $j$th coordinates.
The ordinary moments give $W\subseteq W^\perp$.
Since $B$ is nondegenerate and $\dim W=d+1$, necessarily
$m\ge2d+2$. We therefore only need the equality case, when
$W=W^\perp$.

Translate the sphere to the origin and normalize its radius to one.
The moment and mixed-moment conditions are preserved, and the
weighted energy changes by the square of the radius only.
Write $v_i=D(p_i)$ and
$\delta_i=D(\lambda_i)/(2\lambda_i)$.
Differentiating the zeroth and first moments, and using
$\sum_i\lambda_i v_i=0$, gives
$\delta\in W^\perp=W$.
Set $u_i=v_i+\delta_i p_i$.
Differentiating the quadratic moments, and using symmetry of the
mixed moment matrix, gives
\[
 B(u^{(j)},p^{(k)})=0,
 \qquad B(u^{(j)},\mathbf1)=0.
\]
Thus every $u^{(j)}$ belongs to $W$ as well.
Since $p_i\cdot v_i=0$, we have $p_i\cdot u_i=\delta_i$ and hence
\[
 \sum_i\lambda_i\|v_i\|^2
 =\sum_i\lambda_i\|u_i-\delta_i p_i\|^2
 =\sum_{j=1}^dB(u^{(j)},u^{(j)})-B(\delta,\delta)=0.
\]

For the final assertion, discard zero weights and let $d$ be the
affine dimension of their support. The isotropic-space bound above
gives $2d+2\le m\le6$. A spherical set of affine dimension at most
one has at most two points and admits no nonzero quadratic moment
weights. Thus the only remaining case is $d=2$, $m=6$, which is the
equality case already settled.
\end{proof}

\subsection*{Generic sets}

\begin{theorem}\label{app:thm:generic_sphere}
For $d\ge3$, almost every set of
$m=\binom{d+2}{2}$ points on a fixed $(d-1)$-sphere is not Ramsey.
More precisely, this holds when all $m(d-1)$ stereographic
parameters are algebraically independent over $\Q$.
Among configurations with algebraically independent stereographic
parameters, this point count is the smallest possible for a
nonzero fixed weighted identity from $\|D(q)\|^2$.
\end{theorem}
\begin{proof}
By similarity we work on the unit sphere.
Restrictions of quadratic polynomials to this sphere form a space
of dimension $q=\binom{d+2}{2}-1$.
Their evaluation vectors at any $q$ of the given points are
independent: any failure would be a nonzero polynomial relation
among the stereographic parameters. (The relevant determinant is
not identically zero because the quadratic restrictions are linearly
independent functions.) Thus, with $m=q+1$, there are weights
$\lambda_i\ne0$ annihilating every quadratic polynomial.

Let $V=\bigoplus_{i=1}^m p_i^\perp$ be the space of tangent
velocities. Its dimension is $N=m(d-1)$, and
$H(v,w)=\sum_i\lambda_i v_i\cdot w_i$ is nondegenerate on $V$.
Require
\[
 \sum_i\lambda_i v_i=0,
 \qquad \sum_i\lambda_i v_i p_i^{\mathsf T}
       \text{ is symmetric}.
\]
These impose at most $h=d+\binom d2$ linear conditions.
Since $N>2h$ for $d\ge3$, their common kernel has dimension
greater than $N/2$. It cannot be totally isotropic for $H$.
Consequently it contains $v$ with $H(v,v)\ne0$.
All this linear algebra can be performed over the field generated
by the stereographic parameters, and $v$ can be chosen over that
field too.

The differential of stereographic parametrization identifies
parameter velocities with tangent velocities. Algebraic
independence therefore lets us prescribe a derivation with
$D(p_i)=v_i$ for every $i$. Lemma~\ref{app:lem:criterion}
now proves that the set is not Ramsey.
Algebraic independence holds almost everywhere, as before.

For $m\le q$, the evaluation columns are independent, so all
quadratic moment weights vanish. Necessity of the ordinary moments
then proves the claimed optimality for this form of identity.
\end{proof}

\subsection*{Explicitly attaining the dimension-dependent bound}

The lower bound $2d+3$ for a certificate with all weights nonzero is
attained by the following explicit constructions.  In dimension three
they give the integer-weight nine-point example directly.

For $|z|=1$, put
\[
 \Gamma_k(z)=\frac1{\sqrt{k}}
  (\mathop{\rm Re}z,\mathop{\rm Im}z,
   \mathop{\rm Re}z^2,\mathop{\rm Im}z^2,\ldots,
   \mathop{\rm Re}z^k,\mathop{\rm Im}z^k)\in\R^{2k}.
\]
Thus $\|\Gamma_k(z)\|=1$.  We shall repeatedly use the elementary
interpolation identity
\begin{equation}\label{app:eq:explicit-interpolation}
 \sum_{H(z)=0}\frac{z^r}{H'(z)}=0
 \qquad(0\le r\le \deg H-2),
\end{equation}
when $H$ has distinct roots.

\begin{theorem}\label{app:thm:sharp_dimension}
For every $d\ge2$, there is an explicit affinely $d$-dimensional
spherical set of $2d+3$ points admitting a nonzero fixed weighted
derivation-energy identity with all weights nonzero.  Consequently
the set is not Ramsey.  If $d\ge3$, the set may be chosen so that no
six of its points lie in an affine two-plane.
\end{theorem}
\begin{proof}
First let $d=2k$ be even.  Set $a=\pi-3$, and, for
$1\le j\le2k+1$, define
\[
 y_j=\frac{ja}{2k+2},\qquad
 t_j=\frac{y_j}{\sqrt{1-y_j^2}},\qquad
 z_j=\frac{1+it_j}{1-it_j}.
\]
The $4k+3$ numbers in
\[
             \mathcal Z=\{1\}\cup\{z_j,\overline{z_j}:1\le j\le2k+1\}
\]
are distinct points of the unit circle.  Let
\[
 H(Z)=\prod_{z\in\mathcal Z}(Z-z),\qquad
 \lambda_z=\frac{(z+1)z^{2k}}{H'(z)},
\]
and take the points $\Gamma_k(z)$, $z\in\mathcal Z$.  The product of
the roots of $H$ is one.  Since $\deg H=4k+3$ is odd, direct
conjugation gives
\[
 \overline{H'(z)}=\frac{H'(z)}{z^{4k+1}},
 \qquad \overline{\lambda_z}=\lambda_z.
\]
Thus the weights are real; they are nonzero because $-1\notin\mathcal Z$.

Since $a$ is transcendental, choose a derivation with $D(a)=-a/2$.
If $z=(1+it)/(1-it)$ is one
of the roots above (put $t=0$ for $z=1$), then
\[
 D(t)=-\frac{t(1+t^2)}2,\qquad D(z)=-it z.
\]
Write $b_z=-t$, so that $D(z)=ib_z z$.  Identity
\eqref{app:eq:explicit-interpolation} gives, for $|\ell|\le2k$,
\begin{equation}\label{app:eq:explicit-modes}
 \sum_{z\in\mathcal Z}\lambda_z z^\ell=0,
 \qquad
 \sum_{z\in\mathcal Z}\lambda_z b_z z^\ell=0.
\end{equation}
Indeed, after inserting the definitions, the two sums involve only
the powers $z^{2k+\ell},z^{2k+\ell+1}$, whose exponents lie between
$0$ and $4k+1$.

Every coordinate of $\Gamma_k(z)$ is a linear combination of
$z^r,z^{-r}$ with $1\le r\le k$, and every product of two coordinates
uses only modes of absolute value at most $2k$.  Hence
\eqref{app:eq:explicit-modes} gives all the moment conditions in
Lemma~\ref{app:lem:criterion}; in fact the mixed moment matrix is zero.
Moreover
\[
 \|D\Gamma_k(z)\|^2
   =\frac1k\sum_{r=1}^k r^2 b_z^2.
\]
The corresponding weighted energy does not vanish, because
\begin{align*}
 \sum_{z\in\mathcal Z}\lambda_z b_z^2
 &=\sum_{H(z)=0}
   \frac{-(z-1)^2z^{2k}}{(z+1)H'(z)}
   =\frac4{H(-1)}\ne0.
\end{align*}
For the last equality, divide the numerator by $z+1$, use
\eqref{app:eq:explicit-interpolation} on the polynomial quotient,
and use
$\sum_{H(z)=0}((z+1)H'(z))^{-1}=-1/H(-1)$.
Lemma~\ref{app:lem:criterion} now proves non-Ramseyness.

Now let $d=2k+1$ be odd.  Again put $a=\pi-3$, and set
\[
             c=\frac{a-1}{2},\qquad \rho=a+1,
             \qquad \zeta_j=e^{2\pi i j/(4k)}.
\]
In $\R^{2k}\times\R$, take the following two horizontal layers.  The
upper layer consists of
\[
        (\rho\Gamma_k(\zeta_j),1),\qquad 0\le j<4k,
\]
with weight $(-1)^j$.  In the last coordinate pair of $\R^{2k}$, the
lower layer consists of
\[
 (-1,0),\quad (a,\pm\sqrt{1-a^2}),\quad
 (c,\pm\sqrt{1-c^2}),
\]
with respective weights $-4,-2,-2,4,4$; all its other coordinates
and its final height are zero.  These $4k+5=2d+3$ points lie on the
sphere with centre $(0,\ldots,0,\rho^2/2)$ and squared radius
$1+\rho^4/4$.

Choose $D(a)=1$.  The root-of-unity identity
\[
 \sum_{j=0}^{4k-1}(-1)^j\zeta_j^\ell=0
       \quad (|\ell|\le2k-1),
 \qquad
 \sum_{j=0}^{4k-1}(-1)^j\zeta_j^{\pm2k}=4k
\]
shows that the upper layer has zero weighted mass and first moment,
while its only nonzero quadratic moments form the block
\[
                    2\rho^2\begin{pmatrix}1&0\\0&-1\end{pmatrix}
\]
in the last harmonic coordinate pair.  Its derivative first moment
vanishes, its mixed moment has the same block with $2\rho$ in place
of $2\rho^2$, and its derivative energy is zero.

The lower layer has zero weighted mass, first moment, and derivative
first moment.  Directly using $D(\sqrt{1-u^2})=-uD(u)/\sqrt{1-u^2}$,
its quadratic and mixed moments are the negatives of the two blocks
above.  Its derivative energy is
\[
 -\frac4{1-a^2}+\frac2{1-c^2}
       =-\frac4{(1-a)(3-a)}\ne0.
\]
Thus all the moment conditions hold and the total energy is nonzero;
Lemma~\ref{app:lem:criterion} again applies.  For $k=1$ this is the
explicit integer-weight nine-point construction.

It remains only to record the geometric assertions. A nonzero real
trigonometric polynomial of degree at most $k$ has at most $2k$ zeros
on the circle (multiply its Laurent form by $z^k$ and use the ordinary
polynomial root bound). Equivalently, every at most $2k+1$ distinct points of
$\Gamma_k$ are affinely independent.  This proves affine spanning in
both constructions, and shows that in even dimension $d\ge4$ no four
of the displayed points are coplanar.  In odd dimension, a two-plane
contained in the upper horizontal hyperplane contains at most four
upper points, while a two-plane contained in the lower horizontal
hyperplane contains at most the five points of the lower circle.  A
two-plane contained in neither horizontal hyperplane meets each of
them in at most a line, hence contains at most two points from each
layer.  Thus no six points are coplanar when $d\ge3$.
\end{proof}

\section{Quantitative colour bounds}

The explicit seven-point configuration in Theorem~\ref{thm:main} can
be avoided using twelve colours in every dimension. Indeed, using its
function $F$, colour $q$ by
\[
                         \left\lfloor F(q)/4\right\rfloor\pmod {12}.
\]
The positive and negative weights both have total absolute weight $6$.
If a copy were monochromatic, write $F(p')/4=12k_p+c+\theta_p$, with
$k_p\in\mathbb Z$ and $0\leq\theta_p<1$. Its invariant identity would
give
\[
 -6=12\sum_p\lambda_pk_p+\sum_p\lambda_p\theta_p,
 \qquad -6<\sum_p\lambda_p\theta_p<6,
\]
which is impossible. More generally, integral weights summing to zero,
with positive total $L$, and any nonzero fixed invariant give an
avoiding colouring with $2L$ colours by exactly this argument:
if the invariant is $\sum_p\lambda_pF(p')=E\ne0$, use
$\lfloor LF(q)/E\rfloor\bmod 2L$.
By the compactness theorem for finite hypergraph colourings, the same
finite bound holds on any real inner-product space: every finite collection of forbidden copies lies in a
finite-dimensional Euclidean subspace.

\subsection*{An eight-colour family}

\begin{theorem}\label{app:thm:eight_colours}
There are seven-point subsets of the unit circle that can be avoided
with eight colours in every dimension.
\end{theorem}
\begin{proof}
The following family has weights
$(2,1,1,-1,-1,-1,-1)$, the smallest possible total absolute weight for
seven nonzero integer weights whose sum is zero.

Identify the plane with $\mathbb C$. For any transcendental real number
$t$ with $0<|t|<1/10$ (for example $t=\pi/100$), put
\[
 u=\frac{1+it}{1-it},\quad
 \alpha=\frac{3-u^2}{2},\quad \beta=\frac{3u^2-1}{2},
\]
and consider
\[
 \begin{aligned}
 f_+(z)&=(z-1)^2\left(z^2+\frac{1+u^2}{2}z+u^2\right)
          =z^4-\alpha z^3-\beta z+u^2,\\
 f_-(z)&=z^4-\alpha z^3+\beta z-u^2.
 \end{aligned}
\]
The three distinct roots of $f_+$ and four roots of $f_-$ are the
desired points; the root $1$ has weight $2$, the other roots of $f_+$
have weight $1$, and the roots of $f_-$ have weight $-1$.

We first check the geometric assertions. Since $|u|=1$, writing $z=uw$
in the quadratic factor of $f_+$ gives
\[
                       w^2+\operatorname{Re}(u)w+1=0;
\]
The two roots are distinct, lie on the unit circle, and differ from $1$
in the stated interval. Substituting $z=(1+ix)/(1-ix)$ in $f_-(z)=0$
gives the real quartic equation
\[
 h_t(x)=t(1+6x^2-3x^4)+(1-t^2)x(1-3x^2)=0.
\]
For $0<t<1/10$, its signs at $-\infty,-1,-1/3,0,1$ are respectively
$-,+,-,+,-$. Hence it has four distinct real roots, giving four
distinct unit roots of $f_-$. Negative $t$ follows by conjugation.
If a unit root were shared by $f_+$ and $f_-$, subtraction would
give
\[
                 z=\frac{2u^2}{3u^2-1},
\]
whose modulus is $1$ only if $u^2=1$. Hence the seven points are distinct
throughout the stated interval.

Choose a real derivation with $D(t)=1$ and extend it to $\mathbb C$
by $D(x+iy)=D(x)+iD(y)$. The equal coefficients of $z^3$ and $z^2$ in $f_+$ and
$f_-$ give equality of the first two power sums of their root
multisets. Consequently the weights annihilate the constant, linear,
and quadratic coordinate functions. The weighted product of their
roots is $-1$, because their constant coefficients are $u^2$ and
$-u^2$. For a unit root $z_j$, write $D(z_j)=ib_jz_j$ with $b_j$ real.
Differentiating the product ratio gives
\[
                            \sum_j\lambda_jb_j=0.
\]
Differentiating the weighted first coordinate moments gives
$\sum_j\lambda_jD(p_j)=0$. Since the weights are integers,
differentiating the quadratic moments shows that
$M=\sum_j\lambda_jD(p_j)p_j^{\mathsf T}$ is skew-symmetric.
The preceding angular identity makes its skew-symmetric part zero,
so $M=0$. Thus all moment conditions of the invariant criterion hold.

It remains to check that the energy is nonzero. Let $z_j(t)$ denote
the algebraic root branches and put
$E(t)=\sum_j\lambda_j|z_j'(t)|^2$. At a transcendental parameter $t$,
implicit differentiation gives $D(z_j(t))=z_j'(t)$, so $E(t)$ is the
required invariant energy. The root $1$ of $f_+$ is constant.
At $t=0$, the ordinary angular velocities of its other two roots
are $2,2$, while the angular velocities of the four roots of
$f_-$, ordered as $1,-1,e^{i\pi/3},e^{-i\pi/3}$, are
$-2,2,2,2$. These values follow directly by implicit differentiation;
for example $u'(0)=2i$. Therefore
\[
 \lim_{t\to0}E(t)
                        =2\cdot2^2-4\cdot2^2=-8.
\]
In fact $E(t)$ is a rational function over $\Q(i)$: for each simple
unit root, $|z_j'|^2=-(z_j'/z_j)^2$, and implicit differentiation
expresses $z_j'$ rationally in $t,z_j$. The sum over the roots of
each polynomial is symmetric; for $f_+$ use only its quadratic factor,
since the doubled root $1$ contributes zero. The limit $-8$ shows
that this rational function is not identically zero. It therefore
cannot vanish at a transcendental $t$, and simplicity of the roots
excludes poles on the stated interval. Lemma~\ref{app:lem:criterion}
proves non-Ramseyness, and the preceding
colour bound with $L=4$ gives eight colours.
\end{proof}

\subsection*{A common colouring for many parameters}

Let $S\subset(2,\infty)$ be algebraically independent over $\Q$.
For every dimension there is a single twelve-colouring avoiding all
the configurations $P_r$ from Theorem~\ref{thm:main}, simultaneously
for $r\in S$. Extend $S$ to a transcendence basis and prescribe
\[
                 D(r)=\sqrt{(2r-1)(2r-3)(2r-4)}\quad(r\in S).
\]
The calculation in the main proof now gives the same invariant $-24$
for every $P_r$ with the unscaled function $F(q)=\|D(q)\|^2$.
Consequently $\lfloor F(q)/4\rfloor\bmod12$ avoids them all.
The set $S$ may be chosen dense in $(2,\infty)$ and of cardinality
$|\R|$: choose a countable algebraically independent dense subset,
extend it to a transcendence basis, and move each additional basis
element into $(2,\infty)$ by adding a suitable rational number.

\newpage
\section{Arbitrary quadratic jets in one derivation}

Higher derivatives do not evade the six-point obstruction, even when
their quadratic combination is indefinite.

\begin{theorem}\label{app:thm:all-jets}
Let $P$ be a spherical set of at most six points, let $D$ be a
derivation, and let $(c_{ab})_{0\le a,b\le k}$ be a real symmetric matrix.
Put
\[
 F_n(q)=\sum_{j=1}^n\sum_{a,b=0}^k
                 c_{ab}D^a(q_j)D^b(q_j),\qquad D^0(q_j)=q_j.
\]
If fixed real weights $\lambda_p$ satisfy
$\sum_p\lambda_pF_n(p')=K$ for every congruent copy of $P$ in every
ambient dimension, then $K=0$.
\end{theorem}

\begin{proof}
The zero quadratic form is immediate. After deleting zero rows and columns,
let $k$ be the largest index occurring in the coefficient matrix. If $k=0$
(which also covers $D=0$ after discarding the positive-order terms), then
translation invariance of $c_{00}\sum_p\lambda_p\|p+t\|^2$ forces
$\sum_p\lambda_p=0$ and $\sum_p\lambda_pp=0$. On a centred spherical copy
its value is therefore zero. Hence assume $D\ne0$ and $k\ge1$.
Choose a centred copy whose affine span is a coordinate subspace;
all its derivative vectors remain in that subspace.
We shall use arbitrary finite formal jets of translations and rotations.
This is legitimate: choosing $u$ with $D(u)\ne0$, the matrix
$(D^a(u^b))_{0\le a,b\le m}$ has determinant
$\bigl(\prod_{b=0}^m b!\bigr)D(u)^{m(m+1)/2}$.
Rational combinations of its columns are dense in $\R^{m+1}$.
Thus actual scalar jets are dense among all jets; applying this
entrywise to skew matrices and using the Cayley chart gives the same
statement for orthogonal jets. Polynomial identities therefore extend
to the formal jets used below.

Write $Z_j=\sum_p\lambda_pD^jp$. The quadratic translation terms
give $\sum_p\lambda_p=0$. The linear terms give
$\sum_b c_{ab}Z_b=0$ for every $a$, also after every formal rotation.
Choose $a$ with $c_{ak}\ne0$ and insert
$A(t)=\exp(st^rT/r!)$, where $T$ is skew. The coefficient of $s$ is
\[
       \sum_{b\ge r}\binom br c_{ab}T Z_{b-r}=0.
\]
Taking $r=k,k-1,\ldots,1$ successively, and then using the original
row equation, proves
\begin{equation}\label{app:eq:all-jet-vector-moments}
                     Z_0=Z_1=\cdots=Z_k=0.
\end{equation}

We next extract the two needed rotational conditions. Use the symbol
$c(U,V)=\sum c_{ab}U^aV^b$.
The symbol $(U+V)^m$ represents $D^m\|q\|^2$, so contributes nothing
to a rotational variation. Subtract such symbols, degree by degree.
If nothing remains, the value at every point of the centred sphere is
the same, and $K=0$. Otherwise, let $h$ be the highest nonzero homogeneous
part remaining, of degree $N$. Subtracting a further multiple of
$(U+V)^N$, we may assume
\[
            h(U,0)=h(0,V)=0,\qquad h\ne0.
\]
In particular $N\ge2$.

Let $E$ be the coordinate span of the centred configuration, let $e$ be a
unit normal, and, for $v,w\in E$, define skew maps by
\[
 Sx=e(v\cdot x)-(e\cdot x)v,\qquad
 Tx=e(w\cdot x)-(e\cdot x)w.
\]
Thus, writing $f_p=v\cdot p$ and $g_p=w\cdot p$, we have
$Sp=ef_p$, $Tp=eg_p$, and $(ST+TS)p=-vg_p-wf_p$.
Choose formal scalar jets $\theta^{(j)}(0)=x^j$ and
$\eta^{(j)}(0)=y^j$. To mixed order in $\varepsilon,\delta$, the
orthogonal jet is
\[
 A=I+\varepsilon\theta S+\delta\eta T
   +\tfrac12\varepsilon\delta\theta\eta(ST+TS)
   +O(\varepsilon^2,\delta^2).
\]
This is a finite formal expansion, realizable through the Cayley chart.
Put $B(r,s)=\sum_p\lambda_pr_ps_p$. The product of the two first
variations contributes $2c(U+x,V+y)B(f,g)$, while the two pairings with
the mixed variation contribute
$-c(U+x+y,V)B(f,g)-c(U,V+x+y)B(f,g)$. Hence the
$\varepsilon\delta$ coefficient is the symbol
\[
 2c(U+x,V+y)-c(U+x+y,V)-c(U,V+x+y)
\]
applied to $B(f^{(a)},g^{(b)})$.

The degree-$N$ part in $x,y$ is $2h(x,y)B(f,g)$, because
$h(U,0)=h(0,V)=0$; hence $\sum_p\lambda_ppp^{\mathsf T}=0$.
For degree $N-1$, let
$a=[U^{N-1}V]h=[UV^{N-1}]h$. Taylor expansion gives, up to multiples
of the already-zero $B(f,g)$,
\[
 \bigl(2h_U-a(x+y)^{N-1}\bigr)B(f',g)
 +\bigl(2h_V-a(x+y)^{N-1}\bigr)B(f,g')=0.
\]
Lower homogeneous components contribute at this degree only multiples of
$B(f,g)$. Interchanging $x,y$ and subtracting, using the symmetry of $h$,
gives
\[
 2(h_U-h_V)(x,y)\bigl(B(f',g)-B(f,g')\bigr)=0.
\]
Now $h_U-h_V\ne0$, since otherwise $h$ would be a multiple of
$(U+V)^N$, already removed. As the last parenthesis is
$v^{\mathsf T}(M-M^{\mathsf T})w$ for
$M=\sum_p\lambda_pD(p)p^{\mathsf T}$, this proves that $M$ is symmetric.
Consequently
\begin{equation}\label{app:eq:all-jet-rotation-moments}
 \sum_p\lambda_ppp^{\mathsf T}=0,\qquad
                  \sum_p\lambda_pD(p)p^{\mathsf T}\text{ is symmetric}.
\end{equation}

Delete zero weights, and let $m$ and $d$ be the size and affine dimension
of the remaining support. The space spanned by the constant vector and
the $d$ coordinate vectors is totally isotropic for the nondegenerate
form $B(a,b)=\sum_p\lambda_pa_pb_p$. Hence $d+1\le m/2\le3$.
An affine line meets a sphere in at most two points, so the only
nontrivial possibility is a circle. For any set of at most five distinct
circle points, the evaluation functionals on quadratic polynomials are
independent: separate any one point with a product of two chord equations.
We are therefore left
with six circle points, all weights nonzero.

Place a centred copy of the circle in the first two coordinate directions
and write $p_i=\rho w_i$,
$w_i=(x_i,y_i)$, $\|w_i\|=1$, and $Dw_i=b_iJw_i$.
The space $W=\operatorname{span}(\mathbf1,x,y)$ is now three-dimensional
and totally isotropic in $\R^6$, so $W=W^\perp$.
The first vector moment in \eqref{app:eq:all-jet-vector-moments} and
the symmetry in \eqref{app:eq:all-jet-rotation-moments} say that
$b\in W^\perp$. Thus
\[
                         b_i=\alpha+\beta x_i+\gamma y_i.
\]
For $k=1$, this already gives $B(b,b)=0$, and the weighted sums of
$\|p\|^2$, $\langle p,Dp\rangle$, and $\|Dp\|^2$ all vanish.

Suppose $k\ge2$. The second vector moment also gives
$\sum_i\lambda_iD^2w_i=0$.
Set $z_i=x_i+iy_i$ and $a=(\beta-i\gamma)/2$, extending $D$ by
$D(x+iy)=Dx+iDy$. Then
\[
 Dz_i=ib_iz_i,\qquad b_i=\alpha+az_i+\overline a z_i^{-1}.
\]
In $D^2z_i=(iDb_i-b_i^2)z_i$, the only term of degree higher than two
in $z_i$ is $-2a^2z_i^3$. All remaining terms vanish in the weighted
sum by the quadratic moments. Consequently
\[
                         0=-2a^2\sum_i\lambda_i z_i^3.
\]
The last sum is nonzero: otherwise its conjugate also vanishes, and
the weights annihilate all Laurent monomials of degrees $-3$ through
$3$. Their evaluation matrix at six distinct nonzero points has rank
six by the Vandermonde determinant. Hence $a=0$, so all $b_i$ are equal.
Repeated differentiation now gives
$D^jp_i=A_jw_i+B_jJw_i$, with coefficients independent of $i$.
Every jet inner product is therefore constant across the six points,
and its weighted sum is zero. Thus $K=0$.
\end{proof}

This concerns fixed weighted quadratic identities in one derivation,
not arbitrary finite colourings. In particular, it does not establish
the Ramsey property of any cyclic quadrilateral.

\section{A nonlinear finite-jet barrier}

The preceding quadratic barrier does not become weaker if one merely takes
higher powers of a highest-order jet.  The following observation covers a
fairly broad class of polynomial attempts.

\begin{theorem}\label{app:thm:nonlinear-jet-barrier}
Let $D\colon\R\to\R$ be a nonzero derivation, and let
$P=\{p_1,\ldots,p_n\}$ consist of at most six distinct points on a circle.
Let
\[
 \Phi(X_0,X_1,\ldots,X_k),\qquad X_j\in\R^3,
\]
be a polynomial, where $k\geq1$ is the largest index of a variable occurring
in $\Phi$.  If the degree of $\Phi$ in $X_k$ is at least three, then
there are no weights $\lambda_1,\ldots,\lambda_n$, not all zero, and no
constant $C$ such that
\begin{equation}\label{app:eq:nonlinear-identity}
 \sum_{i=1}^n\lambda_i
 \Phi(q_i,Dq_i,\ldots,D^kq_i)=C,
\end{equation}
for every labelled congruent copy $(q_i)$ of $(p_i)$ in $\R^3$.
\end{theorem}

\begin{proof}
Put $m=\deg_{X_k}\Phi\geq3$, and let
$H(X_0,\ldots,X_{k-1};X_k)$ be the part of $\Phi$ homogeneous of degree $m$
in $X_k$.  Among the coefficients of $H$ as a polynomial in $X_k$, take the
highest total degree $d$ in the lower variables.  By scaling suitable common
formal translation jets of orders below $k$ and extracting the coefficient
of degree $d$, we obtain a nonzero homogeneous polynomial
$h\colon\R^3\to\R$ of degree $m$. Base-copy lower jets contribute only
lower powers of the scaling parameter. Thus $h$ is independent of the
chosen base copy, and the identity obtained below remains valid after any
rigid repositioning of the circle.

We use the same formal-jet realization as in the preceding proof.  In a
rational Cayley chart, prescribe an orthogonal jet which is the identity
through order $k-1$ and whose order-$k$ term is $sS$, where $S$ is an
arbitrary skew-symmetric matrix.  Independently prescribe the order-$k$
translation jet to contribute an arbitrary $u\in\R^3$.  Thus the $k$th jet
of the $i$th point changes by
\[
 X_{k,i}\longmapsto X_{k,i}+u+sSp_i,
\]
while all lower jets remain fixed.  Extracting first the coefficient chosen
above and then the homogeneous degree-$m$ part in $(u,s)$ from
\eqref{app:eq:nonlinear-identity} gives
\begin{equation}\label{app:eq:nonlinear-leading}
 \sum_{i=1}^n\lambda_i h(u+sSp_i)=0,
\end{equation}
for all $u,s,S$.

Choose a vector $\nu$ with $h(\nu)\ne0$, and place the circle, with centre
the origin and radius $\rho$, in the plane perpendicular to $\nu$.  Write
$p_i=\rho(\cos\varphi_i,\sin\varphi_i)$ in that plane.  Taking $u=v\nu$ and
letting $S$ be rotation about an axis in the circle plane at angle $\theta$
gives, up to a common sign,
\[
 Sp_i=\rho\sin(\varphi_i-\theta)\nu.
\]
By the homogeneity of $h$, equation~\eqref{app:eq:nonlinear-leading} therefore implies
\[
 \sum_i\lambda_i
   \bigl(v+s\rho\sin(\varphi_i-\theta)\bigr)^m=0.
\]
Comparing coefficients shows that the weights annihilate
$\sin^j(\varphi_i-\theta)$ for $0\leq j\leq3$ and every $\theta$.  Hence they
annihilate every circle harmonic of order at most three.  With
$z_i=e^{\mathrm i\varphi_i}$, in particular they annihilate the six
consecutive Laurent monomials
\[
 z^{-2},z^{-1},1,z,z^2,z^3.
\]
The corresponding evaluation matrix on any at most six distinct nonzero
$z_i$ has full rank: after multiplying rows by $z_i^2$, it is a Vandermonde
matrix.  Thus every $\lambda_i$ is zero, a contradiction.
\end{proof}

For example, the theorem rules out every potential
$q\mapsto f(\|Lq\|^2)$ in which $L$ is a nonzero linear combination of
$1,D,\ldots,D^k$ with $k\geq1$ and $f$ has degree at least two.  It still
leaves polynomials which are at most quadratic in their highest jet variable
but have larger total degree, genuinely mixed-derivation expressions, and
arbitrary nonpolynomial colourings.  In particular, it gives no answer for
cyclic quadrilaterals.

\section{A limitation of all colourings based only on derivatives}

The following observation gives a different obstruction for a natural
family of cyclic quadrilaterals. It applies to arbitrary functions of
derivative data, not just quadratic potentials.

\begin{theorem}\label{app:thm:algebraic_translations}
Let $P=B\cup\{p_*\}\subset\R^d$, where every coordinate of
every point in $B$ is algebraic. For every $k\ge1$, every $k$-colouring
of $\R^{d(k+1)}$ which is invariant under translations by vectors
with algebraic coordinates contains a monochromatic congruent copy
of $P$.
\end{theorem}
\begin{proof}
Set $q_i=e_i\otimes p_*/\sqrt2$ for $1\le i\le k+1$, identifying
$\R^{d(k+1)}$ with $(\R^d)^{k+1}$.
Two points $q_i,q_j$ have the same colour. The map
\[
 Az=\frac{e_j-e_i}{\sqrt2}\otimes z
\]
is an isometric embedding with algebraic entries. In the copy $q_i+AP$,
each point $q_i+Ab$, $b\in B$, has the colour of $q_i$ by the assumed
translation invariance, and $q_i+Ap_*=q_j$. Hence the copy is
monochromatic.
\end{proof}

Every derivation annihilates the real algebraic numbers. Consequently
a colouring which depends only on the values of any collection of derivations, their positive iterates, or their nonempty compositions is
invariant under algebraic translations. None of these colourings can
exclude, for example, the cyclic quadrilateral
\[
 \left\{(-1,0),(0,1),(1,0),
 \left(\frac{1-t^2}{1+t^2},\frac{2t}{1+t^2}\right)\right\},
 \qquad t\text{ transcendental}.
\]
For a colouring that is an even function of its derivative data,
already the translate $P-p_*/2$ is monochromatic. Its data at every
algebraic vertex are minus one half of the data at $p_*$, while
the exceptional vertex gives plus one half. This includes arbitrary
colourings pulled back through quadratic potentials involving only
positive-order derivative data.
This limitation does not decide whether such quadrilaterals are Ramsey.

\end{document}